\documentclass[12pt]{article}
\usepackage[margin=30mm]{geometry}
\usepackage[colorlinks=true,citecolor=black,linkcolor=black,urlcolor=blue]{hyperref}
\usepackage[capitalise]{cleveref}
\usepackage[numbers,sort&compress]{natbib}
\usepackage{varwidth}
\usepackage{cutwin}
\usepackage{wrapfig}
\usepackage{amssymb}
\usepackage{amsmath}
\usepackage{amsthm}
\usepackage{longtable}
\usepackage{wasysym}
\usepackage{morefloats}
\usepackage{url}
\usepackage{caption}
\usepackage{subcaption}
\usepackage{dsfont}
\usepackage[english]{babel}
\usepackage{multicol}
\usepackage{textcomp}
\usepackage{enumerate}
\usepackage{tikz}
\usetikzlibrary{calc}
\newtheorem{thm}{Theorem}
\newtheorem{lem}[thm]{Lemma}
\newtheorem{cor}[thm]{Corollary}
\newtheorem{conj}[thm]{Conjecture}

\newtheorem{qst}{Question}

\theoremstyle{definition}

\makeatletter

\let\c@table\c@figure
\makeatother 

\renewcommand{\geq}{\geqslant}
\renewcommand{\leq}{\leqslant}

\makeatletter
\renewcommand\section{\@startsection {section}{1}{\z@}%
                                   {-3ex \@plus -1ex \@minus -.2ex}%
                                   {2ex \@plus.2ex}%
                                   {\normalfont\large\bfseries}}
\renewcommand\subsection{\@startsection{subsection}{2}{\z@}%
                                     {-2.5ex\@plus -1ex \@minus -.2ex}%
                                     {1.5ex \@plus .2ex}%
                                     {\normalfont\normalsize\bfseries}}
\renewcommand\subsubsection{\@startsection{subsubsection}{3}{\z@}%
                                     {-2ex\@plus -1ex \@minus -.2ex}%
                                     {1ex \@plus .2ex}%
                                     {\normalfont\normalsize\bfseries}}
 \renewcommand\paragraph{\@startsection{paragraph}{4}{\z@}%
                                    {1.5ex \@plus.5ex \@minus.2ex}%
                                    {-1em}%
                                    {\normalfont\normalsize\bfseries}}
\renewcommand\subparagraph{\@startsection{subparagraph}{5}{\parindent}%
                                       {1.5ex \@plus.5ex \@minus .2ex}%
                                       {-1em}%
                                      {\normalfont\normalsize\bfseries}}
\makeatother

\DeclareMathOperator{\gon}{gon}
\DeclareMathOperator{\clump}{clump}
\DeclareMathOperator{\tw}{tw}
\DeclareMathOperator{\Div}{Div}
\DeclareMathOperator{\supp}{supp}
\def \mod#1{{\:({\rm mod}\ #1)}}

\newcommand\restr[2]{{
  \left.\kern-\nulldelimiterspace 
  #1 
  \right|_{#2} 
  }}

\title{\bf Sparse graphs of high gonality}
\usepackage{textcomp}

\author{Kevin Hendrey\footnote{Research supported by an Australian Postgraduate Award. \texttt{kevin.hendrey@monash.edu}}\\[1ex]
\normalsize School of Mathematical Sciences\\[-0.5ex]
\normalsize Monash University\\[-0.5ex]
\normalsize Melbourne, Australia}

\date{\normalsize\today}

\begin{document}
\maketitle
\begin{abstract}
By considering graphs as discrete analogues of Riemann surfaces, Baker and Norine (Adv. Math. 2007) developed a concept of linear systems of divisors for graphs. Building on this idea, a concept of gonality for graphs has been defined and has generated much recent interest. We show that there are connected graphs of treewidth 2 of arbitrarily high gonality. We also show that there exist pairs of connected graphs $\{G,H\}$ such that $H\subseteq G$ and $H$ has strictly lower gonality than $G$. These results resolve three open problems posed in a recent survey by Norine (Surveys in Combinatorics 2015).
\end{abstract}

\section{Introduction}\label{intro}
The gonality of a curve is an important and well-studied concept in the area of algebraic geometry. Recently, Baker and Norine \cite{BakerNorine07} developed a framework for translating algebraic geometry concepts to analogous graph theory concepts and proved an analogue of the well-known Riemann-Roch Theorem. Their work has led to intensive and fruitful research in this area (see \cite{AminiCaporaso13,KissTothmeresz15,AminiManjunath10,CaporasoLenMelo15,JamesMiranda13} for example). Within the framework they created, a concept of gonality for graphs has been defined and studied \cite{bruynthesis,BruynGijswijt,OmidKool,devjenkaimit}. Notably, Gijswijt  has shown that computing the gonality $\gon(G)$ of a graph $G$ is NP-hard \cite{Gijswijt05}.

The study of graph gonality is in part motivated by possible relationships to other graph parameters. In a recent survey, Norine \cite{NorineSurvey} discusses the potential relevance of gonality to graph minor theory. Central to much of graph minor theory is the graph parameter known as treewidth (see surveys \cite{surveytreewidth,bodlaender98}). Van Dobben de Bruyn and Gijswijt \cite{BruynGijswijt} have shown that the treewidth $\tw(G)$ of a graph $G$  is a lower bound for its gonality, and we know of no connected graph that has been shown to have gonality greater than its treewidth. In his survey, Norine raises the following questions.
\begin{qst}\label{twgonquest}Is there some function $f$ such that for every connected graph $G$, $\gon(G)\leq f(\tw(G))$?\end{qst}

\begin{qst}\label{minorgonquest}Is it true that for every connected graph $G$ and every connected minor $H$ of $G$, $\gon(H)\leq \gon(G)$?\end{qst}
\begin{qst}\label{subgraphgonquest}Is it true that for every connected graph $G$ and every connected subgraph $H$ of $G$, $\gon(H)\leq \gon(G)$?\end{qst}

In Section \ref{treewidthsection} we answer Question \ref{twgonquest} in the negative, proving the following stronger result.
\begin{thm}\label{twgonans}
For all integers $k\geq 2$ and $t\geq k$, there exists a $k$-connected graph $G$ with $\tw(G)=k$ and $\gon(G)\geq t$.
\end{thm} 
In terms of relating connectivity, treewidth and gonality, this result is best possible, as we discuss in Section \ref{treewidthsection}.
We also show that the answer to Question \ref{minorgonquest} is ``no'' by showing that fans have unbounded gonality, while each is a minor of some connected graph of gonality 2. In Section \ref{cyclesection}, we present a class of graphs that have unbounded gonality, while each is a subgraph of some connected graph of gonality 3, thus answering Question \ref{subgraphgonquest} in the negative. However, in the special case where the subgraph $H$ has a universal vertex, we show in Section \ref{treewidthsection} that the answer to Question \ref{subgraphgonquest} is ``yes''. 

 It is interesting to ask which results relating to gonality of curves translate to analogous results in the context of graphs. The following is a well-known result in classical Brill-Noether Theory, due to Griffiths and Harris \cite{GriffithsHarris80}.
\begin{thm}\label{genusgonthm}Every generic curve of genus $g$ has gonality $\lfloor (g+3)/2\rfloor$.\end{thm}
 In the context of graph gonality, the {\it genus} of a connected graph $G$ is the size of a minimum edge set $S$ such that $G-S$ is a tree, and is given by the formula $|E(G)|-|V(G)|+1$.  Baker \cite{Baker08} conjectured the following. 
 \begin{conj}[Gonality Conjecture for Graphs]\label{gonalityconj}For each integer $g\geq 0$:
\begin{enumerate}
\item The gonality of every connected graph of genus $g$ is at most $\lfloor (g+3)/2\rfloor$.
\item There exists a connected graph of genus $g$ and gonality exactly $\lfloor (g+3)/2\rfloor$.
\end{enumerate}
\end{conj}
Part 2 of this conjecture has been proven by Cools, Draisma, Payne and Robeva \cite{coolsdraismapaynerobeva}. Cools and Draisma \cite{coolsdraisma} have also made significant progress on Part 1 of this conjecture. In answering Question \ref{subgraphgonquest} in Section \ref{cyclesection} of this paper, we provide an alternative proof of part 2 of this conjecture.

\section{Preliminaries}\label{gonalitysection}
All graphs in this paper are simple, finite and undirected. Let $G$ be a graph and let $v$ be a vertex of $G$. Let $N(v)$ denote the set $\{w\in V(G):vw\in E(G)\}$ of {\it neighbours} of $v$, and let $N[v]$ denote the set $N(v)\cup \{v\}$. Let $\delta(G)$ denote the minimum degree of $G$. For every positive integer $t$, let $[t]$ denote the set $\{1,2,\dots,t\}$, and let $[0]:=\emptyset$. The operation of {\it contracting} the edge $vw\in E(G)$ consists of deleting $v$ and $w$ and adding a new vertex adjacent to the remaining neighbours of $v$ and the remaining neighbours of $w$. A graph $H$ is a {\it minor} of a graph $G$ if $H$ can be obtained from some subgraph of $G$ by contracting edges. For every positive integer $k$, $G$ is {\it $k$-connected} if $|V(G)|>k$ and for every subset $S\subseteq V(G)$ of size less than $k$, $G-S$ is connected.

The {\it degree matrix} $\mathcal{D}(G)$ of $G$ is the integer matrix in $\mathbb{Z}^{V(G)\times V(G)}$ ($V(G)$ being ordered arbitrarily) with $\mathcal{D}(G)_{v,v}:=\deg(v)$ for all $v\in V(G)$ and $\mathcal{D}(G)_{v,w}:=0$ for all $w\neq v$. The {\it adjacency matrix} $\mathcal{A}(G)$ of $G$ is the  integer matrix in $\mathbb{Z}^{V(G)\times V(G)}$ with $\mathcal{A}(G)_{v,w}:=1$ if $vw\in E(G)$ and $\mathcal{A}(G)_{v,w}:=0$ otherwise. The {\it Laplacian matrix} $Q(G)$ of $G$ is given by $Q(G):=\mathcal{D}(G)-\mathcal{A}(G)$. We write $Q$ for $Q(G)$ when there is no ambiguity. A {\it divisor} of $G$ is a vector in $\mathbb{Z}^{V(G)}$. Let $\Div(G)$ denote the set of divisors of $G$ and for $D\in \Div(G)$ and $v\in V(G)$, let $D(v)$ denote the value of $D$ in position $v$. The {\it support} $\supp(D)$ of $D$ is the set $\{v\in V(G):D(v)\neq 0\}$. For every subgraph $H\subseteq G$, the {\it restriction} $\restr{D}{H}$ of $D$ to $H$ is the divisor in $\Div(H)$ with $\restr{D}{H}(v):=D(v)$ for all $v\in V(H)$. A divisor of $G$ is {\it effective} if each entry is non-negative. Let $\Div_+(G)$ denote the set of effective divisors of $G$. The {\it degree} of a divisor $D$ is given by 
$$\deg(D):=\sum_{v\in V(G)}D(v).$$
Two divisors $D$ and $D'$ are {\it equivalent}, written $D\sim D'$, if there is some divisor $S$ such that $D'=D-QS$. 
Note that each column of $Q$ sums to 0, so $\deg(QS)=0$ for every divisor $S$, and hence every pair of equivalent divisors have the same degree. The {\it rank} $r(D)$ of a divisor $D$ is the maximum value of $k$ such that for every effective divisor $D'$ of degree $k$, there is some effective divisor equivalent to $D-D'$. Note that $r(D)\geq 1$ if and only if for every vertex $v\in V(G)$ there is some effective divisor $D'$ equivalent to $D$ with $D'(v)\geq 1$. The {\it gonality} of a graph, denoted $\gon(G)$, is the minimum degree of a divisor $D$ of $G$ with $r(D)\geq 1$.

Now consider the following chip-firing game for the graph $G$. First, an initial configuration is selected by assigning a non-negative number of chips to each vertex. Making a move in the game consists of selecting a non-empty subset $A\subseteq V(G)$ and for every edge $vw\in E(G)$ with one endpoint $v\in A$ and one endpoint $w\notin A$, moving one chip from $v$ to $w$. In order for a move to be {\it legal}, there must be a non-negative number of chips on every vertex after the move is performed. An initial configuration is {\it winning} if for every vertex $v\in V(G)$, it is possible to transfer a chip to $v$ via some (possibly empty) sequence of legal moves.

There is a natural correspondence between initial configurations and vectors in $\Div_+(G)$. Suppose that $G$ has chip configuration corresponding $D\in \Div_+(G)$, and a move is made by selecting a set $A\subseteq V(G)$. Let $\mathbf{1}_A$ denote the divisor satisfying $\supp(\mathbf{1}_A):=A$ and $\mathbf{1}_A(v):=1$ for all $v\in A$. By the definition of $Q$, the new configuration corresponds to the divisor $D'=D-Q\mathbf{1}_A$. It follows that every winning configuration corresponds to a divisor of rank at least 1.

The following lemma is due to van Dobben de Bruyn and Gijswijt \cite{BruynGijswijt}.
\begin{lem}\label{chainlem}If $D$ and $D'$  are equivalent effective divisors of $G$, then there is a unique chain of non-empty sets $A_1, A_2,\dots, A_t$ and corresponding sequence of divisors $D_0,D_1,\dots, D_t$ such that 
\begin{itemize}
\item $\emptyset \subsetneq A_1\subseteq A_2\subseteq \cdots \subseteq A_t\subsetneq V(G)$, 
\item $D_0=D$ and $D_t=D'$ and
\item for all $i\in [t]$, $D_i$ is effective and $D_i=D_{i-1}-Q\mathbf{1}_{A_i}$.
\end{itemize}
\end{lem}

Now suppose $D$ is a divisor with $r(D)\geq 1$. We may assume that $D$ is effective, since it follows from the fact that $r(D)\geq 0$ that there is some effective divisor in the equivalence class containing $D$. Consider the initial configuration for $G$ corresponding to $D$. Since $r(D)\geq 1$, for every vertex $v$ there is some effective divisor $D'\sim D$ such that $D'(v)\geq 1$. Now Lemma \ref{chainlem} gives a chain of legal moves taking the chip configuration corresponding to $D$ to the chip configuration corresponding to $D'$. Thus, the gonality of a graph can alternatively be defined as the minimum number of chips required for a winning chip configuration.

For several interesting families of graphs, the gonality has been precisely determined by van Dobben de Bruyn and Gijswijt \cite{BruynGijswijt}. In Sections \ref{treewidthsection} and \ref{cyclesection}, we make use of the following result.
\begin{thm}\label{twgonthm}\cite{BruynGijswijt} If $n$ and $m$ are positive integers with $n\geq m$ and $G$ is the $n\times m$ rectangular grid graph, then $\gon(G)=\tw(G)=m$.\end{thm}

\section{Treewidth and Graphs with  a Universal Vertex}\label{treewidthsection}
A vertex $v$ of a graph $G$ is {\it universal} if $N[v]=V(G)$. In this section, we present a formula for the gonality of graphs with a universal vertex. Using this formula, we calculate the gonality of the family of graphs known as fans, consequently proving Theorem \ref{twgonans} and answering Question \ref{minorgonquest}. We also show that the answer to Question \ref{subgraphgonquest} is ``yes'' in the special case where the subgraph $H$ of $G$ has a universal vertex.  

A {\it tree decomposition} of a graph $G$ consists of a tree $T$ together with a function $f$ from $V(T)$ to the set of subsets of $V(G)$, satisfying the following conditions:
\begin{itemize}
\item[1)]for every vertex $v\in V(G)$, there is some vertex $t\in V(T)$ such that $v\in f(t)$,
\item[2)] for every edge $vw\in E(G)$, there is some vertex $t\in V(T)$ such that $\{v,w\}\subseteq f(t)$ and
\item[3)] if $t_3$ is a vertex lying on the path between $t_1$ and $t_2$ in $T$, then $f(t_1)\cap f(t_2)\subseteq f(t_3)$.
\end{itemize}
The width of a tree decomposition is the maximum of $|f(t)|-1$ for $t\in V(T)$, and the {\it treewidth} $\tw(G)$ of a graph $G$ is the minimum width of a tree decomposition of $G$.

In Section \ref{intro}, we claimed that Theorem \ref{twgonans} was best possible. To see this, first note that the claim fails for $k=1$, since every 1-connected graph of treewidth 1 is a tree, and hence has gonality 1. Further, the following elementary result (see \cite{surveytreewidth} for a proof) immediately implies that a $k$-connected graph has treewidth at least $k$.
\begin{lem}\label{mindegtwlem}For every graph $G$, $\tw(G)\geq \delta(G)$.\end{lem}

Due to Lemma \ref{chainlem}, we can focus our analysis on pairs $\{D,D'\}$ of effective divisors such that $D'=D-Q\mathbf{1}_A$ for some set $A$. For this reason, for a graph $G$, an effective divisor $D\in \Div_+(G)$ and a vertex $v\in V(G)$, we are interested in the following set. We define the {\it clump} $\clump(G,D,v)$ centred at $v$ to be the intersection of all subsets $S\subseteq V(G)$ such that $v\in S$ and $D(w)\geq |N(w)\cap S|$ for every vertex $w$ not in $S$. This set is useful due to the following lemma.

\begin{lem}\label{clmplem}If $D$ is an effective divisor of a graph $G$ and $A\subseteq V(G)$ is such that $D-Q\mathbf{1}_A$ is also an effective divisor, then for every subgraph $H$ of $G$ and every vertex $w\in V(H)\setminus A$, we have 
$$\clump(H,\restr{D}{H},w)\subseteq V(H)\setminus A.$$\end{lem}
\begin{proof}
Let $w'$ be a vertex in $V(H)\cap A$. Since $(D-Q\mathbf{1}_A)(w')\geq 0$, we have $D(w')\geq |N(w')\setminus A|$, so $\restr{D}{H}(w')\geq |N_H(w')\setminus A|$. Let $S:=V(H)\setminus A$. Then $S$ is a subset of $V(H)$ such that $w\in S$ and for every vertex $w'\in V(H)\setminus S$, we have  $\restr{D}{H}(w')\geq |N_H(w')\cap S|$. By definition, $\clump(H,\restr{D}{H},w)$ is the intersection of all sets with these properties, so $\clump(H,\restr{D}{H},w)\subseteq V(H)\setminus A$.
\end{proof}

An effective divisor $D$ of a graph $G$ is {\it $v$-reduced} if there is no non-empty subset $A\subseteq V(G)$ such that $v\notin A$ and $D-Q\mathbf{1}_A$ is an effective divisor. In the chip-firing game discussed in Section \ref{gonalitysection}, every legal move from the chip configuration corresponding to a $v$-reduced effective divisor must contain $v$. The following result is due to Baker and Norine \cite{BakerNorine07}.

\begin{lem}\label{vreduced}If $G$ is a connected graph and $D$ is an effective divisor of $G$, then for every vertex $v\in V(G)$, there exists a unique $v$-reduced effective divisor $D'\sim D$.\end{lem} 

The following lemma is our main tool for answering Questions \ref{twgonquest} and \ref{minorgonquest}, and it also answers Question \ref{subgraphgonquest} in the case where the subgraph $H$ of $G$ has a universal vertex.
\begin{lem}\label{localgonlem}Let $G$ be a connected graph, let $H$ be a subgraph of $G$, let $v$ be a vertex of $H$ and let $H':=H-v$. If $v$ is universal in $H$ and $|V(H)|\geq 2$, then
\begin{align*}
&\gon(G)\geq\\&\min\{\max \{\deg(D),\deg(D)+|\clump(H',D,w)|:w\in V(H'), D(w)=0\}:D\in \Div_+(H')\}\\&=\gon(H).
\end{align*}
\end{lem}
\begin{proof}
Let $D$ be an effective divisor of $G$ such that $\deg(D)=\gon(G)$ and $r(D)\geq 1$. By Lemma \ref{vreduced}, we may assume $D$ is $v$-reduced. If $V(H')\subseteq \supp(D)$, then 
\begin{align*}
&\max \{\deg(\restr{D}{H'}),\deg(\restr{D}{H'})+|\clump(H',\restr{D}{H'},w)|: w\in V(H'), D(w)=0\}
\\&=\deg(\restr{D}{H'})\leq \deg(D)= \gon(G).
\end{align*}
If  $V(H')\not\subseteq \supp(D)$, then let $w_0$ be a vertex of $H'$ with $D(w_0)=0$ that maximises $|\clump(H',\restr{D}{H'},w_0)|$. Since $r(D)\geq 1$, there exists $D'\in \Div_+(G)$ such that $D'\sim D$ and $D'(w_0)\geq 1$. Let $A_1,A_2,\dots ,A_t$ and $D_0, D_1,\dots ,D_t$ be defined as in Lemma \ref{chainlem}, with $D_0:=D$ and $D_t:=D'$. Since $D_0$ is $v$-reduced, $v\in A_1$. Since $D_t(w_0)>D_0(w_0)$, there is some $i\in [t]$ such that $w_0\notin A_i$. By definition $A_1\subseteq A_i$, so $w_0\notin A_1$. By Lemma \ref{clmplem} with $A:=A_1$, we have $\clump(H',\restr{D}{H'},w_0)\subseteq V(H')\setminus A_1$. Now, 
\begin{align*}
0\leq D_1(v)= D_0(v)-Q\mathbf{1}_{A_1}(v)= D(v)-|N(v)\setminus A_1|\leq D(v)- |V(H')\setminus A_1|,
\end{align*}
so $D(v)\geq |\clump(H',\restr{D}{H'},w_0)|$ and $\deg(D)\geq \deg(\restr{D}{H'})+|\clump(H',\restr{D}{H'},w_0)|$. Since $\gon(G)=\deg(D)$ and $\restr{D}{H'}\in \Div_+(H')$, we have
\begin{align*}
&\gon(G)\geq \\
&\min \{\max \{\deg(D),\deg(D)+|\clump(H',D,w)|:w\in V(H'), D(w)=0\}:D\in \Div_+(H')\}.
\end{align*}

Now suppose $E\in \Div_+(H)$ is such that $E(v)=\max\{0,|\clump(H',\restr{E}{H'},w)|:w\in V(H'), E(w)=0\}$, and subject to this $\deg(E)$ is minimised. If $V(H')\subseteq \supp(E)$, then $r(E)\geq 1$, since $E':=E-Q(H)\mathbf{1}_{V(H')}$ is an effective divisor equivalent to $E$ such that $E'(v)\geq 1$, since $v$ is not the only vertex of $H$. If $V(H')\not\subseteq \supp(E)$, then let $w_0\in V(H')$ be such that $E(w_0)=0$. Now $E(v)\geq |\clump(H',\restr{E}{H'},w_0)|$ and $w_0\in\clump(H',\restr{E}{H'},w_0)$, so $E(v)\geq 1$. Let $A:=V(H)\setminus\clump(H',\restr{E}{H'},w_0)$, and consider the divisor $E':=E-Q(H)\mathbf{1}_A$. We have $E'(v)= E(v)-|\clump(H',\restr{E}{H'},w_0)|\geq 0$ and for all $w'\in \clump(H',\restr{E}{H'},w_0)$, we have $E'(w')\geq E(w')+1\geq 1$ since $v\in N(w')\cap A$. Suppose $w'\in A\setminus \{v\}$. By the definition of $\clump(H',\restr{E}{H'},w_0)$, there is some set $S\subseteq V(H')$ such that $w'\notin S$, $w_0\in S$ and for every vertex $w''\in V(H')\setminus S$, we have $E(w'')\geq|N(w'')\cap S|$. Since $w'\notin S$, we have $E(w')\geq |N(w')\cap S|$, and since $\clump(H',\restr{E}{H'},w_0)\subseteq S$, we have $E(w')\geq|N(w')\cap \clump(H',\restr{E}{H'},w_0)|$. Hence $E'(w')\geq 0$. Therefore, $E'$ is an effective divisor equivalent to $E$ such that $E'(w')\geq 1$, so $r(E)\geq 1$, and $\gon(H)\leq \deg(E)$. By our choice of $E$,
\begin{align*}
&\deg(E)=\\&\min\{\max \{\deg(D),\deg(D)+|\clump(H',D,w)|:w\in V(H'), D(w)=0\}:D\in \Div_+(H')\}.
\end{align*}
Hence, by our previous result with $G:=H$, we have $\gon(H)=\deg(E)$. 
\end{proof}

The {\it fan} on $n$ vertices is the $n$-vertex graph with a universal vertex $v$ such that $G-v$ is a path. It is well-known that fans have treewidth 2. Lemma \ref{localgonlem} provides a method for determining the gonality of a fan. The following lemma is the final tool we need.

\begin{lem}\label{emptyclump}Let $G$ be a graph and let $D$ be an effective divisor of $G$. If $v\in V(G)$ and $H:=G[\{v\}\cup \{w\in V(G):D(w)=0\}]$, then the vertex set of the component of $H$ that contains $v$ is a subset of $\clump(G,D,v)$.\end{lem}
\begin{proof}
By definition, $v\in \clump(G,D,v)$. Let $t$ be a non-negative integer, and suppose for induction that every vertex at distance exactly $t$ from $v$ in $H$ is in $\clump(G,D,v)$, and suppose $w_0$ is at distance exactly $t+1$ from $v$ in $H$. Since $w_0\in V(H-v)$, we have $D(w_0)=0$, and since $w_0$ is at distance $t+1$ from $v$ in $H$, $w_0$ has some neighbour $w_1$ in $H$ at distance exactly $t$ from $v$ in $H$. By our inductive hypothesis, $w_1\in \clump(G,D,v)$. Hence, for every subset $S\subseteq V(G)$ such that $v\in S$ and $D(w)\geq |N(w)\cap S|$ for every vertex $w$ not in $S$, we have $|N(w_0)\cap S|\geq |\{w_1\}|>D(w_0)$, so $w_0\in S$. Therefore $w_0\in \clump(G,D,v)$, and by induction every vertex in the component of $H$ that contains $v$ is in $\clump(G,D,v)$.
\end{proof}

\begin{thm}\label{fangonlem}If $G$ is the fan of $n$ vertices, then 
$$\gon(G)=\min\{t+\lceil (n-1-t)/(t+1)\rceil:t\in\{\lfloor \sqrt{n}-1\rfloor, \lceil\sqrt{n}-1\rceil\}\}.$$\end{thm}
\begin{proof}
Let $v$ be the universal vertex in $G$, and let $G-v=p_1p_2\cdots p_{n-1}$. For every integer $t\geq 0$, let $f(t):=\min\{\max\{0,|\clump(G-v,D,w)|: w\in N(v), D(w)=0\}:D\in \Div_+(G-v), \deg(D)=t\}$. By Lemma \ref{localgonlem}, $\gon(G)= \min\{t+f(t):t\in \mathbb{Z}, t\geq 0\}$. Let $D$ be an effective divisor of $G-v$, let $H(D):=G[\{w\in N(v):D(w)=0\}]$, and let $w_0$ be a vertex of $H(D)$. By Lemma \ref{emptyclump}, the vertex set of the component subpath $P$ of $H(D)$ containing $w_0$ is a subset of $\clump(G-v,D,w_0)$. Every neighbour $w'$ of $P$ in $G-v$ satisfies $D(w')\geq 1$, since $P$ is a component of $H(D)$. Since $G-v$ is a path, $|N(w')\cap P|\leq 1$. Hence, $w_0\in V(P)$ and $D(w')\geq |N(w')\cap V(P)|$ for every vertex $w'\in N(v)$, so $\clump(G-v,D,w_0)=V(P)$ by the definition of $\clump(G-v,D,w_0)$. Hence, $f(t)=\min \{\max\{|V(P)|:P \textrm{ is a component of }H(D)\}:D\in \Div_+(G-v)\}$. Since $H(D)$ is entirely determined by $\supp(D)$, we may restrict ourselves to divisors satisfying $D=\mathbf{1}_{\supp(D)}$. In particular, $\gon(G)= \min\{t+f(t):t\in \{0,1,\dots,n-1\}\}$. For all $t\in \{0,1,\dots ,n-1\}$ and all $D\in \Div_+(G-v)$ with $\deg(D)=t$, the graph $H(D)$ has at most $t+1$ components and at least $n-1-t$ vertices, so $f(t)\geq \lceil (n-1-t)/(t+1)\rceil$. Let $A_t:=\{p_k:k/(\lceil (n-1-t)/(t+1)\rceil+1)\in [t]\}$ and let $D:=\mathbf{1}_{A_t}$. Then $D\in \Div_+(G-v)$, $\deg(D)=t$ and $\max\{|V(P)|:P \textrm{ is a component of }H(D)\}=\lceil (n-1-t)/(t+1)\rceil$. Hence, $f(t)=\lceil (n-1-t)/(t+1)\rceil$. Consider the function $g$ on the domain $(-1,\infty)$ given by $g(x):=x+(n-1-x)/(x+1)$, and note that for $t\in \{0,1,\dots, n-1\}$, we have $t+f(t)=\lceil g(t)\rceil$. Since $g(x)$ has a unique minimum at $x=\sqrt{n}-1$, has no local maxima and is continuous, $t+f(t)$ is minimised for $t\in \{0,1,\dots, n-1\}$ either at $t=\lfloor \sqrt{n}-1\rfloor$ or at $\lceil\sqrt{n}-1\rceil$.
\end{proof}
We can now prove Theorem \ref{twgonans} and answer Question \ref{minorgonquest}.

\begin{proof}[Proof of Theorem \ref{twgonans}]
Let $n:=(t+2)^2$, and let $G$ be graph formed by adding $k-1$ universal vertices $v_1,v_2,\dots ,v_{k-1}$ to the path $P:=p_1p_2\cdots p_{n-1}$. Let $T:=P-p_{n-1}$, and let $f$ be the function from $V(T)$ to the set of subsets of $V(G)$ such that $f(p_i)=\{p_i,p_{i+1},v_1,v_2,\dots, v_{k-1}\}$ for $i\in [n-2]$. It is quick to check that $T$ and $f$ form a tree-decomposition of $G$ of width $k$, so $\tw(G)\leq k$. Since $\delta(G)=k$, we have $\tw(G)= k$ by Lemma \ref{mindegtwlem}. Let $H:=G- \{v_2,v_3,\dots,v_{k-1}\}$, and note that $H$ is the fan on $n$ vertices with universal vertex $v_1$. By Lemma \ref{localgonlem}, we have $\gon(G)\geq \gon(H)$. By Lemma \ref{fangonlem}, 
\begin{align*}
\gon(H)= \min\{t+\lceil (n-1-t)/(t+1)\rceil:t\in\{\lfloor \sqrt{n}-1\rfloor, \lceil\sqrt{n}-1\rceil\}\}> \sqrt{n}-2=t.\\ \qedhere
 \end{align*} 
\end{proof}

\begin{cor}\label{minorgonans}For every integer $t\geq 0$ there exist connected graphs $G$ and $H$ such that $H$ is a minor of $G$, $\gon(G)=2$ and $\gon(H)> t$.\end{cor}
\begin{proof}
Let $n:=(t+2)^2$, let $G$ be the $2\times (n-1)$ rectangular grid and let $H$ be the fan on $n$ vertices. Then $H$ can be obtained from $G$ by contracting one of the rows of $G$ to a single vertex, so $H$ is a minor of $G$, and $\gon(G)=2$ by Theorem \ref{twgonthm}. By Theorem \ref{fangonlem}, we have 
\begin{align*}
\gon(H)= \min\{t+\lceil (n-1-t)/(t+1)\rceil:t\in\{\lfloor \sqrt{n}-1\rfloor, \lceil\sqrt{n}-1\rceil\}\}> \sqrt{n}-2=t.\\ \qedhere
 \end{align*} 
\end{proof}

\section{Blocks and Block-Cut Vertex Trees}\label{cyclesection}

In this section we show that the answer to Question \ref{subgraphgonquest} is ``no'', thus providing an alternative proof that the answer to Question \ref{minorgonquest} is ``no''. In the process of doing this, we also provide an alternative proof of part 2 of Conjecture \ref{gonalityconj}.

A subgraph $H$ of a connected graph $G\not\cong K_1$ is a {\it block} of $G$ if there is no 2-connected subgraph $H'$ of $G$ with $H$ as a proper subgraph, and either $H$ is 2-connected or $H\cong K_2$. Each edge of a graph is in exactly one block, and the edge of a block that is isomorphic to $K_2$ is a {\it bridge}. A vertex $v\in V(G)$ is a {\it cut vertex} if $G-v$ is disconnected. Let $\mathcal{C}(G)$ be the set of cut vertices of a graph $G$, and let $\mathcal{B}(G)$ be the set of blocks of $G$. The bipartite graph with vertex partition $(\mathcal{C}(G),\mathcal{B}(G))$ such that $v\in \mathcal{C}(G)$ adjacent to $B\in \mathcal{B}(G)$ if $v\in V(B)$ is known as the block-cut vertex tree of $G$.

Our approach will be to construct graphs in which every block is a cycle, for which the block-cut tree is a path. Lemmas \ref{blockgonlem} and \ref{reducedmaxlem} are our main tools for determining the gonality of a graph based on the structure of its blocks. We combine these with Lemma \ref{cyclegonlem}, which characterises the equivalence classes of divisors of cycles, to determine the gonality of the graphs we construct.
\begin{lem}\label{blockgonlem}Let $B$ be a block of a connected graph $G$, and let $w_0$ and $w_1$ be vertices in $V(B)$, not necessarily distinct. If $D$ is a $w_0$-reduced effective divisor of $G$, then
$$\max\{D'(w):D'\in\Div_+(G),D'\sim D\}=\max\{D'(w):D'\in\Div_+(B),D'\sim\restr{D}{B}\}.$$
\end{lem}

\begin{proof}
Let $D'\in \Div_+(G)$ be such that $D'\sim D$ and $D'(w)$ is maximised. Let $D_0,D_1,\dots,D_t$ and $A_1,A_2,\dots ,A_t$ be defined as in Lemma \ref{chainlem}, with $D_0:=D$ and $D_t:=D'$. Let $E_0:=\restr{D}{B}$, and for $i\in [t]$, let $B_i:=A_i\cap V(B)$, and let $E_i:=E_{i-1}-Q(B)\mathbf{1}_{B_i}$. Suppose for contradiction that for some vertex $v\in V(G)$ and some $k\in [t]$, $E_k(v)<D_k(v)$ and $E_{k}(v)\geq D_{k}(v)$. We cannot have $v\in A_k$, since the fact that $(N(v)\cap V(B))\setminus B_k\subseteq N(v)\setminus A_k$ would mean $E_{k-1}(v)-E_k(v)\leq D_{k-1}(v)-D_k(v)$, a contradiction. Hence, $v\notin A_k$, and $v$ has some neighbour $v'$ in $A_k\setminus V(B)$. Let $C$ be the component of $G-v$ containing $v'$, let $k'\in [k]$ be the smallest number such that $A_{k'}\cap V(C)$ is non-empty, let $A':=A_{k'}\cap V(C)$ and let $D'':=D-Q\mathbf{1}_{A'}$. Since $v\notin A_k$ and $k'\leq k$, we have $v\notin A_{k'}$ by the definition of $A_1,A_2,\dots,A_t$. Hence, for $w\in A'$, we have $N(w)\setminus A'=N(w)\setminus A_k$. We also have $N[w]\cap A_t=\emptyset$ for all $t\in[k'-1]$, so $D(w)=D_{k'-1}(w)$. Hence, $D''(w)=D_{k'}(w)\geq 0$. For $w\notin A'$ we have $D''(w)\geq D(w)\geq 0$. Hence, $D''\in \Div_+(G)$. But $w_0\notin A'$, contradicting the assumption that $D$ is $w_0$-reduced. Now, for all $v\in V(G)$ and all $i\in [t]$, either $E_i(v)\geq D_i(v)$ or $E_{i-1}(v_0)<D_{i-1}(v_0)$. Since $E_0(v)=D_0(v)$, we have $E_t(v)\geq D_t(v)=D'(v)$ by induction. Hence, $E_t$ is an effective divisor equivalent to $\restr{D}{B}$ with $E_t(w)\geq D'(w)$.

Let $E'$ be an effective divisor equivalent to $\restr{D}{B}$ that maximises $E'(w)$, and let $S\in \Div(B)$ be such that $E'=\restr{D}{B}-Q(B)S$. Let $S'\in \Div(G)$ be the divisor such that $\restr{S'}{B}=S$ and $\restr{S'}{B'}$ is single valued for every block $B'$ of $G$ other than $B$. Then $D':=D-QS'$ is an effective divisor of $G$ equivalent to $D$ such that $D'(w)=E'(w)$.
\end{proof}

\begin{lem}\label{reducedmaxlem}If $G$ is a connected graph, $v\in V(G)$ and $D$ is a $v$-reduced effective divisor of $G$, then $D(v)=\max\{D'(v):D'\in \Div_+(G), D\sim D'\}$.\end{lem}
\begin{proof}
Let $D'$ be an effective divisor equivalent to $D$ such that $D'(v)$ is maximised, and suppose for contradiction that $D'(v)>D(v)$. Let $A_1,A_2,\dots ,A_t$ and $D_0,D_1,\dots ,D_t$ be defined as in Lemma \ref{chainlem}, with $D_0:=D$ and $D_t:=D'$. Since $D'(v)>D(v)$, there is some $i\in [t]$ such that $v\notin A_i$. But then $v\notin A_1$ and $D_1:=D_0-Q\mathbf{1}_{A_1}$ is an effective divisor, contradicting the assumption that $D$ is $v$-reduced.
\end{proof}

We now characterise the equivalence classes of divisors of cycles. For the following proof, we define the notion of a chip function. If $D\in\Div_+(G)$ and $\mathcal{C}$ is a set of size $\deg(D)$, then a {\it $\mathcal{C}$-chip function} for $D$ is a function from $\mathcal{C}$ to $V(G)$ such that $|\{c\in \mathcal{C}:f_1(c)=v\}|=D(v)$ for all $v\in V(G)$.
\begin{lem}\label{cyclegonlem}Let $C_t$ be the graph with $V(C_b):=\mathbb{Z}_b$ and edge set $\{xy:x-y\in \{-1,1\}\}$. If $D$ and $D'$ are divisors of $C_t$, then let $\Delta(D,D')$ be given by  
$$\Delta(D,D'):=\sum_{v\in V(C_t)}(D_0(v)-D_1(v)).$$
Effective divisors $D_0$ and $D_1$ of $C_t$ of equal degree are equivalent if and only if $\Delta(D_0,D_1)=0$.
\end{lem} 
\begin{proof}
Note that for $w\in V(G)$, if $D$ and $D'$ are divisors of $G$ such that $D':=D-Q\mathbf{1}_{\{w\}}$, then $\Delta(D,D')=2w-(w+1)-(w-1)=0$. Every divisor $S\in \Div(G)$ can be expressed as a sum of integer multiples of divisors in $\{\mathbf{1}_{\{w\}}:w\in V(G)\}$, so if $D_0\sim D_1$, then $\Delta(D_0,D_1)=0$.

Now suppose for contradiction that $\Delta(D_0,D_1)=0$ and $D_0$ and $D_1$ are not equivalent. Let $k:=\deg(D_0)=\deg(D_1)$, let $\mathcal{C}:=\{c_1,c_2,\dots , c_k\}$ and let $f_1$ be a $\mathcal{C}$-chip function for $D_1$. Select an effective divisor $D_2$ equivalent to $D_0$ and a $\mathcal{C}$-chip function $f_2$ for $D_2$ such that the size of the set $S:=\{c\in \mathcal{C}:f_1(c)\neq f_2(c)\}$ is minimised, and subject to this, the minimum over $S$ of the distance from $f_1(c)$ to $f_2(c)$ is minimised. Note that $\Delta(D_2,D_1)$ is  equal to the sum over $c\in S$ of $f_2(c)-f_1(c)$. Now, from the definition of $\Delta$, we have $\Delta(D_2,D_1)=\Delta(D_2,D_0)+\Delta(D_0,D_1)$, which is 0 since $D_2\sim D_0$. Hence, the sum over $S$ of $f_1(c)-f_2(c)$ is 0, and since $D_0$ and $D_1$ are not equivalent, $|S|\geq 2$. Without loss of generality, $c_1$ is an element of $S$ that minimises the distance between $f_1(c_1)$ and $f_2(c_1)$ and $c_2$ is also in $S$. We have $f_2(c_2)\neq f_1(c_1)$, or else we could improve $f_2$ by swapping the values of $f_2(c_1)$ and $f_2(c_2)$. Let $P$ be the path in $C_t$ with end points $f_2(c_1)$ and $f_2(c_2)$ that does not contain the vertex $f_1(c)$. Let $D_3:=D_2-Q\mathbf{1}_{V(P)}$ and let $f_3$ be the $\mathcal{C}$-chip function for $D_3$ such that for $i\in\{3,4,\dots,k\}$, $f_3(c_i):=f_2(c_i)$  and for $i\in\{1,2\}$, $f_3(c_i)$ is the neighbour of $f_2(c_i)$ that is not in $P$. Then $\{c\in \mathcal{C}:f_1(c)\neq f_3(c)\}\subseteq S$ and the distance from $f_3(c_1)$ to $f_1(c_1)$ is strictly less than the distance from $f_2(c_1)$ to $f_1(c_1)$, a contradiction.
\end{proof}

\begin{lem}
\label{cyclechaingonlem}
Let $G$ be a graph with $g\geq 3$ blocks, each of which is a cycle of size $b\geq 3$, such that the block-cut tree of $G$ is a path and every pair of distinct cut vertices $v$ and $w$ that share a block are at distance exactly $k\geq 1$ from each other. Then $$\gon(G)=\min(\lfloor (g+3)/2\rfloor, \min\{t\in [b]:tk\equiv 0 \mod{b}\}).$$
\end{lem}
\begin{proof}
Label the blocks of $G$ with $B_0, B_1,\dots ,B_{g-1}$ so that $B_i$ intersects $B_j$ if and only if $|i-j|\leq 1$, and for $i\in [g-1]$ let $v_i$ be the vertex at the intersection of $B_{i-1}$ and $B_i$. Let $v_g$ be a vertex in $B_{g-1}$ at distance exactly $k$ from $v_{g-1}$. Let $D$ be a $v_1$ reduced effective divisor of $G$ of rank at least 1. Suppose for contradiction that for some $i\in [g-1]$, $\deg(\restr{D}{B_i- v_i})\geq 2$. If $D(v)\geq 2$ for some $v\in V(B_i- v_i)$, then set $A:=\{v\}$. Otherwise, there are distinct vertices $v$ and $w$ in $V(B_i-v_i)$ such that $D(v)=D(w)=1$. Let $P$ be the path from $v$ to $w$ in $B_i-v_i$, and set $A:=V(P)$. If $v_{i+1}\in A$, then set $A':=A\cup V(G_1)$, where $G_1$ is the connected component of $G-v_{i+1}$ containing $v_g$, unless $i+1=g$, in which case $G_1$ is the empty graph. If $v_{i+1}\notin A$, then set $A':=A$. Now $D-Q\mathbf{1}_{A'}\in \Div_+(G)$, contradicting the assumption that $D$ is $v_1$ reduced. Hence, for $i\in [g-1]$, $\deg(\restr{D}{B_i- v_i})\in\{0,1\}$. 

For $i\in [g]$, define $f(i)$ by 
$$f(i):=\max\{D'(v_i):D'\in \Div_+(G),D'\sim D\}.$$
Fix $i\in [g-1]$. By Lemma \ref{vreduced} there is a unique $v_i$-reduced effective divisor $D'$ equivalent to $D$, and by Lemma \ref{reducedmaxlem}, $D'(v_i)=f(i)$. Let $D_0,D_1,\dots ,D_t$ and $A_1,A_2,\dots ,A_t$ be defined as in Lemma \ref{chainlem}, with $D_0:=D$ and $D_t:=D'$. Suppose for contradiction that $v_i\in A_j$ for some $j\in [t]$. Then $v_i\in A_t$ and $D_{t-1}=D_t-Q\mathbf{1}_{A_t^C}$, contradicting the assumption that $D'$ is $v_i$-reduced. Hence  $v_0$ is not in $A_j$ for any $j\in [t]$.

  Let $G_1$ be the component of $G- v_i$ containing $v_g$, and suppose for contradiction that $A_j$ intersects $G_1$ for some $j\in [t]$, and let $j_0$ be the minimum element of $[t]$ such that $A_{j_0}$ intersects $G_1$. Let $A':=A_{j_0}\cap V(G_1)$, and let $D'':=D-Q\mathbf{1}_{A'}$. Consider an arbitrary vertex $v\in V(G_1)$.  Since $j_0$ is minimum and $v_0$ is not in $A_j$ for any $j\in [t]$, we have $D_{j_0-1}(v)=D(v)$. Hence, we have $D''(v)=D_{j_0}(v)\geq 0$. Consider an arbitrary vertex $v\notin V(G_1)$. Since $v\notin A'$, we have $D''(v)\geq D(v)\geq 0$. Therefore $D''$ is an effective divisor of $G$ equivalent to $D$, contradicting the assumption that $D$ is $v_1$-reduced. Hence, for $j\in [t]$, we have $A_j\cap(V(G_1)\cup \{v_i\})=\emptyset$. Hence, for $v\in V(G_1)$, we have $D'(v)=D(v)$.

By Lemma \ref{blockgonlem} with $B:=B_i$, $w_0:=v_i$ and $w:=v_{i+1}$, $f(i+1)=\max\{D''(v_{i+1}):D''\in \Div_+(B_i),D''\sim \restr{D'}{B_i}\}$. Let $C_b$ be the graph with $V(C_b):=\mathbb{Z}_b$ and $E(C_b):=\{xy:x-y\in\{1,-1\}\}$. Let $\theta$ be an isomorphism from $B_i$ to $C_b$ such that $\theta(v_i)=0$ and $\theta(v_{i+1})=k$. Let $F\in \Div_+(C_b)$ be the divisor such that $F(\theta(w)):=D'(w)$ for $w\in V(B_i)$, and let $F'\in \Div_+(C_b)$ be an effective divisor with $\deg(F')=\deg(F)$. By Lemma \ref{cyclegonlem}, we have $F\sim F'$ if and only if $\Delta(F,F')=0$.

Suppose first that $\deg(\restr{D}{B_i- v_i})=0$. Then $F(x):=0$ for $x\neq 0$. Now 
\[\Delta(F,F')=\sum_{v\in V(C_b)}F(v)v-F'(v)v=\sum_{v\in V(C_b)}-F'(v)v.\]
Since $\deg(F')=\deg(F)=f(i)$, we know $F'(k)\leq f(i)$. If $\supp(F'):=\{k\}$ and $F'(k):=f(i)$, then $\Delta(F,F')=0$ if and only if $kf(i)=0$, so $f(i+1)\leq f(i)$, with equality if and only if  $kf(i)=0$.

Suppose instead that $\deg(\restr{D}{B_i- v_i})\neq 0$. Recall that $\deg(\restr{D}{B_i- v_i})\in \{0,1\}$, so we have $\deg(\restr{D}{B_i- v_i})=1$. Now $\deg(F')=\deg(F)=f(i)+1$, so $F'(k)\leq f(i)+1$, and hence $f(i+1)\leq f(i)+1$. We now know the following:
\begin{align}
f(i+1) 
\begin{cases}
  = f(i)& \text{if}\ kf(i)\equiv 0\mod{b}  \text{ and} \deg(\restr{D}{B_i- v_i})=0, \\
  \leq f(i) +2\deg(\restr{D}{B_i- v_i})-1           & \text{if}\ kf(i)\not\equiv 0\mod{b}.
\end{cases}
\end{align}

Suppose $D$ satisfies $\supp(D):=\{v_1\}$ and $D(v_1):=\min\{t\in [b]:tk\equiv 0 \mod{b}\}$. Now $f(1)=\min\{t\in [b]:tk\equiv 0 \mod{b}\}$, and by $(1)$, we have $f(i)=\min\{t\in [b]:tk\equiv 0 \mod{b}\}$ for $i\in [g+1]$. Let $w$ be a vertex of $G$, and let $B_i$ be a block containing $w$. By Lemma \ref{vreduced} there is a $v_i$ reduced effective divisor $D'$ equivalent to $D$, and by Lemma \ref{reducedmaxlem}, $D'(v_i)=f(i)$. Since $b\geq 3$ and $1\leq k< b$, we have $f(i)=\min\{t\in [b]:tk\equiv 0 \mod{b}\}\geq 2$. Since $B_i$ is a cycle, by Lemma \ref{blockgonlem} there is some effective divisor $D''$ equivalent to $D'$ with $D''(w)\geq 1$. Hence $r(D)\geq 1$, and $\gon(G)\leq \min\{t\in [b]:tk\equiv 0 \mod{b}\}$. 

Suppose $\lfloor(g+3)/2\rfloor<\min\{t\in [b]:tk\equiv 0 \mod{b}\}$, and $D$ is the $v_1$-reduced effective divisor equivalent to the divisor $D'$ such that $\supp (D'):=\{v_{\lfloor g/2\rfloor}\}$ and $D'(v_{\lfloor g/2\rfloor}):=\lfloor(g+3)/2\rfloor$. Since $\deg(D)<\min\{t\in [b]:tk\equiv 0 \mod{b}\}$, for all $i\in [g]$ we have $kf(i)\not\equiv 0\mod{b}$. Hence, by $(1)$, for all $i\in [\lfloor g/2 \rfloor-1]$, we have $f(i)\geq f(i+1)-1$. By construction, $f(\lfloor g/2\rfloor)=\lfloor(g+3)/2\rfloor$. Hence
for $i\in [\lfloor g/2\rfloor]$, we have $f(i)\geq \lfloor (g+3)/2\rfloor-(\lfloor g/2 \rfloor-1)\geq 2$. As before, for all $i\in \{0,1,\dots ,\lfloor g/2\rfloor\}$ and each vertex $w\in V(B_i)$, there is some effective divisor $D''$ equivalent to $D$ with $D''(w)\geq 1$. Now let $\phi$ be an automorphism for $G$ such that $\phi(v_{\lfloor g/2\rfloor})=v_{\lceil g/2\rceil}$ and $\phi(v_1)=\phi(v_{g-1})$, and let $\phi(D')$ be given by $\phi(D')(v):=D'(\phi(v))$. Redefine $D$ to be the $v_1$-reduced effective divisor of $G$ equivalent to $\phi(D')$, and redefine $f$ accordingly. By the same argument as before, for $i\in [\lceil g/2\rceil]$, we have $f(i)\geq \lfloor (g+3)/2\rfloor-(\lceil g/2\rceil-1)\geq 2$.  As before, for all $i\in \{\lfloor(g+1)/2\rfloor, \lfloor(g-1)/2\rfloor+1,\dots,g-2, g-1\}$ and each vertex $w\in V(B_i)$, there is some effective divisor $D''$ equivalent to $D$ with $D''(w)\geq 1$. For every vertex $w\in V(G)$, either $w\in V(B_i)$ for some $i\in\{0,1,\lfloor g/2\rfloor\}$ or $\phi(w)\in V(B_i)$ for some $i\in \{0,1,\dots ,\lceil g/2\rceil\}$. Hence, $r(D')\geq 1$, and $\gon(G)\leq\lfloor (g+3)/2\rfloor$.

Suppose that $\deg(D)<\min(\lfloor (g+3)/2\rfloor, \min\{t\in [b]:tk\equiv 0 \mod{b}\})$. If for some $i\in[g]$, $f(i)=0$, then $r(D)<1$. Otherwise for $i\in [g]$, we have $1\leq f(i)\leq \deg(G)$, so $kf(i)\not\equiv 0\mod{b}$. Hence, $f(g)\leq f(1)+2(\deg(\restr{D}{G-B_0}))-(g-1)\leq 2\deg(D)-\deg(\restr{D}{B_0})-g+1$. If $\deg(\restr{D}{B_0})\leq 1$, then $r(\restr{D}{B_0})< 1$, so by Lemma \ref{blockgonlem}, $r(D)< 1$. Otherwise, note that  $2\deg(D)<2\lfloor (g+3)/2\rfloor$, so $f(g)\leq (g+1)-2-g+1=0$ and $r(D)<1$. Hence, $\gon(G)\geq \min(\lfloor (g+3)/2\rfloor, \min\{t\in [b]:tk\equiv 0 \mod{b}\})$.
\end{proof}

By carefully selecting values for $b$, Lemma \ref{cyclechaingonlem} allows us to construct a graph of genus $G$ and gonality $\lfloor (g+3)/2\rfloor$ for all $g\geq 3$. For $g\in \{0,1,2\}$, the $2\times (g+1)$ rectangular grid has genus g, and has gonality $\lfloor (g+3)/2\rfloor$ by Theorem \ref{twgonthm}. Thus our results provide an alternative method for proving part 2 of Conjecture \ref{gonalityconj}. 

We now answer Question \ref{subgraphgonquest}.


\begin{cor}\label{subgraphgonans} For every positive integer $t_0$, there exist connected graphs $G$ and $H$ such that $\gon(G)=t_0$, $\gon(H)=3$ and $G\subseteq H$.\end{cor}
\begin{proof}
When $t_0\leq 3$, the result is trivial, so suppose $t_0\geq 4$. Let $G$ be the graph with $2t_0-3$ blocks, each of which is a cycle of size $2t_0$, such that the block-cut forest of $G$ is a path and every pair of distinct cut vertices $v$ and $w$ that share a block are at distance exactly $t_0-1$ from each other. Then, by Lemma \ref{cyclechaingonlem}, $\gon(G)=\min(t_0, \min\{t\in [2t_0]:(t_0-1)t\equiv 0 \mod{2t_0}\})$. The least common  multiple of $t_0-1$ and $2t_0$ is either $t_0(t_0-1)$ or $2t_0(t_0-1)$ depending on whether $t_0$ is even or odd, so $\min\{t\in [2t_0]:(t_0-1)t\equiv 0 \mod{2t_0}\}\geq t_0$ and $\gon(G)=t_0$. Let $k:=(2t_0-3)(t_0-1)+1$, and let $H$ be the $3\times k$ rectangular grid. As illustrated in Figure \ref{3gonsubgraphpic}, $G\subseteq H$. By Theorem \ref{twgonthm}, $\gon(H)=3$, as required.
\end{proof}

\begin{center}

\begin{tikzpicture}[line width=1pt,vertex/.style={circle,inner sep=0pt,minimum size=0.2cm}] 

    \pgfmathsetmacro{\n}{3}; 
    \pgfmathsetmacro{\m}{\n-1};

  \node[draw=black,fill=gray] (01) at ($(0,0)$) [vertex] {};
  \node[draw=black,fill=gray] (02) at ($(\n,0)$) [vertex] {};
  \node[draw=black,fill=gray] (03) at ($(2*\n,0)$) [vertex] {};
  \node[draw=black,fill=gray] (04) at ($(3*\n,0)$) [vertex] {};
  \node[draw=black,fill=gray] (05) at ($(4*\n,0)$) [vertex] {};
 \node[draw=black,fill=gray] (06) at ($(5*\n,0)$) [vertex] {};
    \node[draw=black,fill=gray] (11) at ($(0,1)$) [vertex] {};
  \node[draw=black,fill=gray] (12) at ($(\n,1)$) [vertex] {};
  \node[draw=black,fill=gray] (13) at ($(2*\n,1)$) [vertex] {};
  \node[draw=black,fill=gray] (14) at ($(3*\n,1)$) [vertex] {};
  \node[draw=black,fill=gray] (15) at ($(4*\n,1)$) [vertex] {};
  \node[draw=black,fill=gray] (16) at ($(5*\n,1)$) [vertex] {};

  \node[draw=black,fill=gray] (22) at ($(\n,2)$) [vertex] {};
  \node[draw=black,fill=gray] (23) at ($(2*\n,2)$) [vertex] {};
  \node[draw=black,fill=gray] (24) at ($(3*\n,2)$) [vertex] {};
  \node[draw=black,fill=gray] (25) at ($(4*\n,2)$) [vertex] {};
  \node[draw=gray,fill=none] (26) at ($(5*\n,2)$) [vertex] {};
  \node[draw=black,fill=gray] (1a) at ($(0.33*\n,1)$) [vertex] {};
  \node[draw=black,fill=gray] (1b) at ($(0.67*\n,1)$) [vertex] {};
   \node[draw=black,fill=gray] (2a) at ($(1.33*\n,1)$) [vertex] {};
  \node[draw=black,fill=gray] (2b) at ($(1.67*\n,1)$) [vertex] {};
   \node[draw=black,fill=gray] (3a) at ($(2.33*\n,1)$) [vertex] {};
  \node[draw=black,fill=gray] (3b) at ($(2.67*\n,1)$) [vertex] {};
   \node[draw=black,fill=gray] (4a) at ($(3.33*\n,1)$) [vertex] {};
  \node[draw=black,fill=gray] (4b) at ($(3.67*\n,1)$) [vertex] {};
     \node[draw=black,fill=gray] (5a) at ($(4.33*\n,1)$) [vertex] {};
  \node[draw=black,fill=gray] (5b) at ($(4.67*\n,1)$) [vertex] {};
  
  \node[draw=black,fill=gray] (2a2) at ($(1.33*\n,2)$) [vertex] {};
  \node[draw=black,fill=gray] (2b2) at ($(1.67*\n,2)$) [vertex] {};

    \node[draw=black,fill=gray] (3a0) at ($(2.33*\n,0)$) [vertex] {};
  \node[draw=black,fill=gray] (3b0) at ($(2.66*\n,0)$) [vertex] {};

    \node[draw=black,fill=gray] (4a2) at ($(3.33*\n,2)$) [vertex] {};
  \node[draw=black,fill=gray] (4b2) at ($(3.67*\n,2)$) [vertex] {};

      \node[draw=black,fill=gray] (5a0) at ($(4.33*\n,0)$) [vertex] {};

  \node[draw=black,fill=gray] (1b0) at ($(0.67*\n,0)$) [vertex] {};

   \node[draw=black,fill=gray] (1a0) at ($(0.33*\n,0)$) [vertex] {};
    \node[draw=black,fill=gray] (5b0) at ($(4.67*\n,0)$) [vertex] {};

    \node[draw=gray,fill=none] (21) at ($(0,2)$) [vertex] {};
      \node[draw=gray,fill=none] (1a2) at ($(0.33*\n,2)$) [vertex] {};
      \node[draw=gray,fill=none] (1b2) at ($(0.67*\n,2)$) [vertex] {};

       \node[draw=gray,fill=none] (2a0) at ($(1.33*\n,0)$) [vertex] {};
       \node[draw=gray,fill=none] (2b0) at ($(1.67*\n,0)$) [vertex] {};

         \node[draw=gray,fill=none] (3a2) at ($(2.33*\n,2)$) [vertex] {};
         \node[draw=gray,fill=none] (3b2) at ($(2.67*\n,2)$) [vertex] {};

           \node[draw=gray,fill=none] (4a0) at ($(3.33*\n,0)$) [vertex] {};
           \node[draw=gray,fill=none] (4b0) at ($(3.67*\n,0)$) [vertex] {};

             \node[draw=gray,fill=none] (5a2) at ($(4.33*\n,2)$) [vertex] {};
             \node[draw=gray,fill=none] (5b2) at ($(4.67*\n,2)$) [vertex] {};

  
\draw[draw=black, line width=2pt](01)--(11);
\draw[draw=black, line width=2pt](02)--(12);
\draw[draw=black, line width=2pt](03)--(13);
\draw[draw=black, line width=2pt](04)--(14);
\draw[draw=black, line width=2pt](05)--(15);
\draw[draw=black, line width=2pt](06)--(16);
\draw[draw=black, line width=2pt](13)--(23);
\draw[draw=black, line width=2pt](14)--(24);

\draw[draw=black, line width=2pt](12)--(22);
\draw[draw=black, line width=2pt](15)--(25);
\draw[draw=gray, style=dashed](16)--(26);
\draw[draw=black, line width=2pt](01)--(1a0);
\draw[draw=black, line width=2pt](11)--(1a);

\draw[draw=black, line width=2pt](12)--(1b);

\draw[draw=black,line width=2pt](1b)--(1a);

\draw[draw=black, line width=2pt](12)--(2a);
\draw[draw=black, line width=2pt](13)--(2b);

\draw[draw=black,line width=2pt](2a)--(2b);

\draw[draw=black, line width=2pt](13)--(3a);
\draw[draw=black, line width=2pt](14)--(3b);

\draw[draw=black,line width=2pt](3a)--(3b);

\draw[draw=black, line width=2pt](14)--(4a);
\draw[draw=black, line width=2pt](15)--(4b);

\draw[draw=black, line width=2pt](4a)--(4b);

\draw[draw=black, line width=2pt](15)--(5a);
\draw[draw=black, line width=2pt](16)--(5b);

\draw[draw=black, line width=2pt](5a)--(5b);

\draw[draw=black, line width=2pt](22)--(2a2);
\draw[draw=black, line width=2pt](23)--(2b2);

\draw[draw=black,line width=2pt](2a2)--(2b2);

\draw[draw=black, line width=2pt](03)--(3a0);
\draw[draw=black, line width=2pt](04)--(3b0);

\draw[draw=black, line width=2pt](3a0)--(3b0);

\draw[draw=black, line width=2pt](24)--(4a2);
\draw[draw=black, line width=2pt](25)--(4b2);

\draw[draw=black, line width=2pt](4a2)--(4b2);

\draw[draw=black, line width=2pt](05)--(5a0);
\draw[draw=black, line width=2pt](02)--(1b0);

\draw[draw=black, line width=2pt](1a0)--(1b0);

\draw[draw=black, line width=2pt](5a0)--(5b0);



\draw[draw=gray, style=dashed](21)--(1a2);
\draw[draw=gray, style=dashed](22)--(1b2);

\draw[draw=gray, style=dashed](1a2)--(1b2);

\draw[draw=gray, style=dashed](02)--(2a0);
\draw[draw=gray, style=dashed](03)--(2b0);

\draw[draw=gray, style=dashed](2a0)--(2b0);

\draw[draw=gray, style=dashed](23)--(3a2);
\draw[draw=gray, style=dashed](24)--(3b2);

\draw[draw=gray, style=dashed](3a2)--(3b2);

\draw[draw=gray, style=dashed](04)--(4a0);
\draw[draw=gray, style=dashed](05)--(4b0);

\draw[draw=gray, style=dashed](4a0)--(4b0);

\draw[draw=gray, style=dashed](25)--(5a2);

\draw[draw=gray, style=dashed](5a2)--(5b2);

\draw[draw=gray, style=dashed](11)--(21);
\draw[draw=gray, style=dashed](1a2)--(1a);
\draw[draw=gray, style=dashed](1a)--(1a0);
\draw[draw=gray, style=dashed](1b)--(1b0);
\draw[draw=gray, style=dashed](1b)--(1b2);

\draw[draw=gray, style=dashed](2a2)--(2a);
\draw[draw=gray, style=dashed](2a)--(2a0);
\draw[draw=gray, style=dashed](2b)--(2b0);
\draw[draw=gray, style=dashed](2b)--(2b2);

\draw[draw=gray, style=dashed](3a2)--(3a);
\draw[draw=gray, style=dashed](3a)--(3a0);
\draw[draw=gray, style=dashed](3b)--(3b0);
\draw[draw=gray, style=dashed](3b)--(3b2);

\draw[draw=gray, style=dashed](4a2)--(4a);
\draw[draw=gray, style=dashed](4a)--(4a0);
\draw[draw=gray, style=dashed](4b)--(4b0);
\draw[draw=gray, style=dashed](4b)--(4b2);

\draw[draw=gray, style=dashed](5a2)--(5a);
\draw[draw=gray, style=dashed](5a)--(5a0);

\draw[draw=gray, style=dashed](5b)--(5b0);
\draw[draw=gray, style=dashed](5b)--(5b2);

\draw[draw=gray, style=dashed](5b2)--(26);
\draw[draw=black, line width=2pt](5b0)--(06);

\end{tikzpicture}
\captionsetup{hypcap=false}

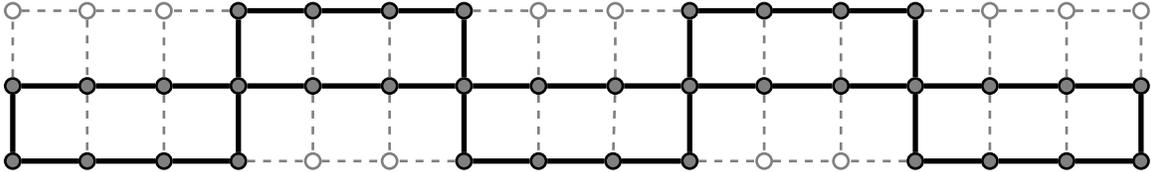
\captionof{figure}{$G$, in black, as a subgraph of $H$, when $t_0=4$.}\label{3gonsubgraphpic}
\end{center}

If the aim is to minimise the gonality of the supergraph $H\supseteq G$ in Corollary \ref{subgraphgonans}, then we can improve this construction slightly by replacing each cut vertex of $G$ with a bridge, maintaining the path structure of the block-cut vertex tree. For sufficiently large $N$, the resulting graph would be a subgraph of the $2\times N$ rectangular grid (as in Figure \ref{2gonsubgraphpic}). By considering the chip-firing game mentioned in Section \ref{gonalitysection}, it is clear that this operation does not affect the gonality. This can be seen by first noting that chips can be transferred back and forth between the endpoints of a bridge without moving chips to or from any other vertex, and then noting that for every bridge, every legal move can be broken down into two separate legal moves, one in which no chip is transferred across the bridge and one in which no chips move except for possibly one chip being transferred from one endpoint of the bridge to the other. Since the $2\times N$ rectangular grid has gonality 2 by Theorem \ref{twgonthm}, there are gonality-2 graphs with subgraphs of arbitrarily high gonality.

\begin{center}

\begin{tikzpicture}[line width=1pt,vertex/.style={circle,inner sep=0pt,minimum size=0.2cm}] 

    \pgfmathsetmacro{\n}{3};

  \node[draw=black,fill=gray] (00) at ($(0,0)$) [vertex] {};
  \node[draw=black,fill=gray] (10) at ($(0.75*\n,0)$) [vertex] {};
 \node[draw=black,fill=gray] (20) at ($(\n,0)$) [vertex] {};
 \node[draw=black,fill=gray] (30) at ($(1.75*\n,0)$) [vertex] {};
\node[draw=black,fill=gray] (40) at ($(2*\n,0)$) [vertex] {};
  \node[draw=black,fill=gray] (50) at ($(2.75*\n,0)$) [vertex] {};
 \node[draw=black,fill=gray] (60) at ($(3*\n,0)$) [vertex] {};
 \node[draw=black,fill=gray] (70) at ($(3.75*\n,0)$) [vertex] {};
  \node[draw=black,fill=gray] (80) at ($(4*\n,0)$) [vertex] {};
 \node[draw=black,fill=gray] (90) at ($(4.75*\n,0)$) [vertex] {};

   \node[draw=black,fill=gray] (01) at ($(0,1)$) [vertex] {};
  \node[draw=black,fill=gray] (11) at ($(0.75*\n,1)$) [vertex] {};
 \node[draw=black,fill=gray] (21) at ($(\n,1)$) [vertex] {};
 \node[draw=black,fill=gray] (31) at ($(1.75*\n,1)$) [vertex] {};
\node[draw=black,fill=gray] (41) at ($(2*\n,1)$) [vertex] {};
  \node[draw=black,fill=gray] (51) at ($(2.75*\n,1)$) [vertex] {};
 \node[draw=black,fill=gray] (61) at ($(3*\n,1)$) [vertex] {};
  \node[draw=black,fill=gray] (71) at ($(3.75*\n,1)$) [vertex] {};
  \node[draw=black,fill=gray] (81) at ($(4*\n,1)$) [vertex] {};
 \node[draw=black,fill=gray] (91) at ($(4.75*\n,1)$) [vertex] {};
 
   \node[draw=black,fill=gray] (00a) at ($(0.25*\n,0)$) [vertex] {};

 \node[draw=black,fill=gray] (00d) at ($(0.5*\n,0)$) [vertex] {};

   \node[draw=black,fill=gray] (20a) at ($(1.25*\n,0)$) [vertex] {};

 \node[draw=black,fill=gray] (20d) at ($(1.5*\n,0)$) [vertex] {};

   \node[draw=black,fill=gray] (40a) at ($(2.25*\n,0)$) [vertex] {};

 \node[draw=black,fill=gray] (40d) at ($(2.5*\n,0)$) [vertex] {};

   \node[draw=black,fill=gray] (01a) at ($(0.25*\n,1)$) [vertex] {};

 \node[draw=black,fill=gray] (01d) at ($(0.5*\n,1)$) [vertex] {};

   \node[draw=black,fill=gray] (21a) at ($(1.25*\n,1)$) [vertex] {};

 \node[draw=black,fill=gray] (21d) at ($(1.5*\n,1)$) [vertex] {};
 
    \node[draw=black,fill=gray] (41a) at ($(2.25*\n,1)$) [vertex] {};

 \node[draw=black,fill=gray] (41d) at ($(2.5*\n,1)$) [vertex] {};
 
     \node[draw=black,fill=gray] (60a) at ($(3.25*\n,0)$) [vertex] {};
 
 \node[draw=black,fill=gray] (60d) at ($(3.5*\n,0)$) [vertex] {};
 
     \node[draw=black,fill=gray] (61a) at ($(3.25*\n,1)$) [vertex] {};
       \node[draw=black,fill=gray] (61d) at ($(3.5*\n,1)$) [vertex] {};
       \node[draw=black,fill=gray] (81a) at ($(4.25*\n,1)$) [vertex] {};
       \node[draw=black,fill=gray] (81d) at ($(4.5*\n,1)$) [vertex] {};
   \node[draw=black,fill=gray] (80a) at ($(4.25*\n,0)$) [vertex] {};
       \node[draw=black,fill=gray] (80d) at ($(4.5*\n,0)$) [vertex] {};
 
\draw[draw=black, line width=2pt](00)--(01);
\draw[draw=black, line width=2pt](10)--(11);
\draw[draw=black, line width=2pt](20)--(21);
\draw[draw=black, line width=2pt](30)--(31);
\draw[draw=black, line width=2pt](40)--(41);
\draw[draw=black, line width=2pt](50)--(51);
\draw[draw=black, line width=2pt](60)--(61);

\draw[draw=black, line width=2pt](00)--(00a);
\draw[draw=black, line width=2pt](00d)--(10);

\draw[draw=black, line width=2pt](00a)--(00d);

\draw[draw=black, line width=2pt](20)--(20a);
\draw[draw=black, line width=2pt](20d)--(30);

\draw[draw=black, line width=2pt](20a)--(20d);

\draw[draw=black, line width=2pt](40)--(40a);
\draw[draw=black, line width=2pt](40d)--(50);

\draw[draw=black, line width=2pt](40a)--(40d);

\draw[draw=black, line width=2pt](01)--(01a);
\draw[draw=black, line width=2pt](01d)--(11);

\draw[draw=black, line width=2pt](01a)--(01d);

\draw[draw=black, line width=2pt](21)--(21a);
\draw[draw=black, line width=2pt](21d)--(31);
\draw[draw=black, line width=2pt](21a)--(21d);

\draw[draw=black, line width=2pt](41)--(41a);
\draw[draw=black, line width=2pt](41d)--(51);
\draw[draw=black, line width=2pt](41a)--(41d);

\draw[draw=black, line width=2pt](60)--(60a);
\draw[draw=black, line width=2pt](60d)--(70);

\draw[draw=black, line width=2pt](60a)--(60d);

\draw[draw=black, line width=2pt](61)--(61a);
\draw[draw=black, line width=2pt](61d)--(71);

\draw[draw=black, line width=2pt](61a)--(61d);

\draw[draw=black, line width=2pt](80)--(80a);
\draw[draw=black, line width=2pt](80d)--(90);

\draw[draw=black, line width=2pt](80a)--(80d);

\draw[draw=black, line width=2pt](81)--(81a);
\draw[draw=black, line width=2pt](81d)--(91);

\draw[draw=black, line width=2pt](81a)--(81d);

\draw[draw=black, line width=2pt](10)--(20);
\draw[draw=black, line width=2pt](30)--(40);
\draw[draw=black, line width=2pt](50)--(60);

\draw[draw=gray, style=dashed](11)--(21);
\draw[draw=gray, style=dashed](31)--(41);
\draw[draw=gray, style=dashed](51)--(61);

\draw[draw=gray, style=dashed](00a)--(01a);
\draw[draw=gray, style=dashed](00d)--(01d);
\draw[draw=gray, style=dashed](20a)--(21a);
\draw[draw=gray, style=dashed](20d)--(21d);
\draw[draw=gray, style=dashed](40a)--(41a);
\draw[draw=gray, style=dashed](40d)--(41d);
\draw[draw=gray, style=dashed](60a)--(61a);

\draw[draw=gray, style=dashed](71)--(81);
\draw[draw=black, line width=2pt](70)--(80);
\draw[draw=gray, style=dashed](80d)--(81d);
\draw[draw=gray, style=dashed](60d)--(61d);
\draw[draw=gray, style=dashed](80a)--(81a);
\draw[draw=black, line width=2pt](70)--(71);
\draw[draw=black, line width=2pt](90)--(91);
\draw[draw=black, line width=2pt](80)--(81);
\end{tikzpicture}
\captionsetup{hypcap=false}

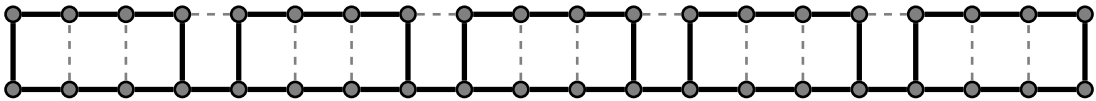
\captionof{figure}{A modification of the graph $G$ in Figure \ref{3gonsubgraphpic} as a subgraph of a rectangular grid.}\label{2gonsubgraphpic}
\end{center}
\section*{Acknowledgements}
The author wishes to acknowledge and thank Darcy Best, Thomas Hendrey, Timothy Wilson and David Wood for helpful suggestions and discussions.
\bibliographystyle{myNatbibStyle}
\bibliography{phdbib}
\end{document}